\documentclass[12pt]{amsart}
\usepackage[margin=1.25in]{geometry}

\usepackage{amsmath, amscd}
\usepackage{mathrsfs}
\usepackage{amsfonts}
\usepackage{latexsym,epsfig}
\usepackage{mathabx}
\usepackage{amsmath, amscd, amsthm}
\usepackage[pdftex]{color}
\usepackage{comment}
\usepackage{marginnote}
\usepackage[colorlinks,citecolor=blue,linkcolor=red]{hyperref}

\newcommand{\N}{\mathbb{N}}

\newcommand{\Out}[0]{\mathrm{Out}}

\renewcommand{\P}{\mathbb{P}}

\usepackage{hyperref}

\newtheorem{theorem}{Theorem}[section]
\newtheorem{lemma}[theorem]{Lemma}

\newtheorem{corollary}[theorem]{Corollary}

\newtheorem{question}[theorem]{Question}
\theoremstyle{definition}
\newtheorem{remark}[theorem]{Remark}

\newcommand{\I}{\mathcal{I}}

\newcommand{\T}{\mathcal{T}}

\newcommand{\C}{\mathcal{C}}

\newcommand{\Mod}{\mathrm{Mod}}
\newcommand{\Aut}{\mathrm{Aut}}

\begin{document}

\title{Random extensions of free groups and surface groups are hyperbolic}

\author[S.J. Taylor]{Samuel J. Taylor}
\address{Department of Mathematics, 
Yale University, 
20 Hillhouse Ave, 
New Haven, CT 06520, USA}
\email{\href{mailto:s.taylor@yale.edu}{s.taylor@yale.edu}}

\author[G. Tiozzo]{Giulio Tiozzo}
\address{Department of Mathematics, 
Yale University, 
20 Hillhouse Ave, 
New Haven, CT 06520, USA}
\email{\href{mailto:giulio.tiozzo@yale.edu}{giulio.tiozzo@yale.edu}}
\date{\today}

\begin{abstract}
In this note, we prove that a random extension of either the free group $F_N$ of rank $N\ge3$ or of the fundamental group of a closed, orientable surface $S_g$ of genus $g\ge2$ is a hyperbolic group. Here, a random extension is one corresponding to a subgroup of either $\Out(F_N)$ or $\Mod(S_g)$ generated by $k$ independent random walks. Our main theorem has several applications, including that a random subgroup of a weakly hyperbolic group is free and undistorted.
\end{abstract}

\maketitle

\section{Introduction}
Let $F=F_N$ denote the free group of rank $N\ge 3$ and $S=S_g$ the closed, orientable surface of genus $g\ge2$. Group extensions of both $F$ and $\pi_1(S)$ can be understood by investigating subgroups of their respective outer automorphism groups. For this, denote the outer automorphism group of $F$ by $\Out(F)$ and the mapping class group of $S$ by $\Mod(S)$. 

For the surface $S$, there is the well-known Birman exact sequence \cite{birman1969mapping}
$$1 \to \pi_1(S) \to \Mod(S;p) \overset{f}{\to} \Mod(S) \to 1, $$
where $\Mod(S;p)$ denotes the group of mapping classes that fix the marked point $p \in S$ and $f: \Mod(S;p) \to \Mod(S)$ is the surjective homomorphism that forgets this requirement. Given any finitely generated subgroup $\Gamma \le \Mod(S)$, its preimage $E_\Gamma = f^{-1}(\Gamma)$ in $\Mod(S;p)$ is a finitely generated group fitting into the sequence
$$1 \to \pi_1(S) \to E_\Gamma \to \Gamma \to 1. $$
We say that $E_\Gamma$ is the \emph{surface group extension} corresponding to $\Gamma \le \Mod(S)$. Much work has gone into understanding what conditions on $\Gamma \le \Mod(S)$ imply that the corresponding extension $E_\Gamma$ is hyperbolic. Such subgroups were introduced by Farb and Mosher as \emph{convex cocompact} subgroups of the mapping class group \cite{FarbMosher} and have since become an active area of study. See for example \cite{KentLein,KLSurvey,H, DKL1, MT1} and Section \ref{sub:hyp_surface} for details.

The situation for extensions of the free group $F$ is similar; by definition there is the short exact sequence
$$1 \to F \to \Aut(F) \overset{f}{\to} \Out(F) \to1, $$
where $f: \Aut(F) \to \Out(F)$ is now the induced quotient homomorphism. As before, a finitely generate subgroup $\Gamma \le \Out(F)$ pulls back via $f$ to the corresponding \emph{free group extension} $E_\Gamma = f^{-1}(\Gamma)$. Conditions on $\Gamma \le \Out(F)$ which imply that the extension group $E_\Gamma$ is hyperbolic were recently given by Dowdall and the first author in \cite{DT1}. See Section \ref{sub:hyp_free} for details.

In this note, we consider the following question: 
\begin{question}\label{q}
Given a  \emph{random} subgroup $\Gamma$ of either $\Mod(S)$ or $\Out(F)$, how likely is it that the corresponding extension group $E_\Gamma$ is hyperbolic?
\end{question}

By employing techniques developed by Maher and the second author in \cite{MaherTiozzo}, we answer Question \ref{q} by considering subgroups generated by random walks on either $\Mod(S)$ or $\Out(F)$. The point is that the references above characterize hyperbolicity of the extension group $E_\Gamma$ (for both $F$ and $\pi_1(S)$) solely in terms of the action of $\Gamma$ on a certain hyperbolic graph. This is precisely the situation considered in 
\cite{MaherTiozzo}. For this reason, we can treat both the cases of free group extensions and surface group extensions at once. 

\subsection{Results}
To state our main theorem, we briefly introduce our model of random subgroups.  
Additional background on random walks is given in Section \ref{sub:random}. 

First, let $X$ be a separable hyperbolic metric space and $G$ a countable group acting on $X$ by isometries. The action $G \curvearrowright X$ is said to be \emph{nonelementary} if there are $g,h \in G$ which are loxodromic for the action and whose quasiaxis determine $4$ distinct endpoints on $\partial X$. A probability measure $\mu$ on $G$ is \emph{nonelementary} if the semigroup generated by its support is a subgroup of $G$ whose action on $X$ is nonelementary.

Now let $\mu$ be a nonelementary probability measure on $G$ with respect to the action $G\curvearrowright X$ and consider $k$ independent random walks $(w^1_n)_{n \in \N}, \ldots, (w^k_n)_{n \in \N}$ whose increments are distributed according to $\mu$. For each $n\in \N$, we can consider the subgroup generated by the $n^{th}$ steps of our random walks,
$$\Gamma(n) = \langle w^1_n, \ldots, w^k_n \rangle \le G, $$
which we endow with the word metric coming from a given generating set. 

\begin{theorem}\label{th:random_qi}
Let $G$ be a countable group with a nonelementary action by isometries on a separable hyperbolic space $X$. Let $\mu$ be a nonelementary probability measure on $G$ and fix $x_0 \in X$. Then, the probability that the orbit map 
\begin{align*}
\Gamma(n) = \langle w^1_n, \ldots, w^k_n \rangle &\to X \\
g &\mapsto g\cdot x_0 
\end{align*}
is a quasi--isometric embedding goes to $1$ as $n \to \infty$. 
\end{theorem}

Combining Theorem \ref{th:random_qi} with work of Farb--Mosher \cite{FarbMosher}, Kent--Leininger \cite{KentLein}, and Hamenst\"adt \cite{H} (see Section \ref{sub:hyp_surface}) answers Question \ref{q} for surface group extensions:

\begin{theorem}[Random surface group extensions] \label{th:surface_extensions}
Let $\mu$ be a nonelementary probability measure on $\Mod(S)$ and let $\Gamma(n) = \langle w^1_n, \ldots, w^k_n \rangle$ denote the subgroup generated by the $n^{th}$ steps of $k$ independent random walks. Then the probability that the surface group extension $E_{\Gamma(n)}$ is hyperbolic goes to $1$ as $n \to \infty$.
\end{theorem}
\begin{proof}
As discussed in Section \ref{sub:hyp_surface}, the mapping class group has a nonelementary action by isometries on the curve graph $\C(S)$. Theorem \ref{th:random_qi} then implies that the probability that the orbit map  $\Gamma(n) \to \C(S)$ is a quasi--isometric embedding goes to $1$ as $n \to \infty$. Since the extension $E_{\Gamma(n)}$ is hyperbolic whenever $\Gamma(n) \to \C(S)$ is a quasi--isometric embedding (Theorem \ref{th:cc}), the theorem follows.
\end{proof}

For extensions of free groups, we answer Question \ref{q} by combining Theorem \ref{th:random_qi} with work of Dowdall and the first author \cite{DT1} (see Section \ref{sub:hyp_free}):

\begin{theorem}[Random free group extensions] \label{th:free_extensions}
Let $\mu$ be a nonelementary probability measure on $\Out(F)$ and let $\Gamma(n) = \langle w^1_n, \ldots, w^k_n \rangle$. Then the probability that the free group extension $E_{\Gamma(n)}$ is hyperbolic goes to $1$ as $n \to \infty$.
\end{theorem}
\begin{proof}
As in Section \ref{sub:hyp_free}, $\Out(F)$ has a nonelementary action by isometries on the hyperbolic graph $\I$. The remainder of the proof follows exactly as in Theorem \ref{th:surface_extensions} after replacing $\C(S)$ with $\I$ and using Theorem \ref{th:free_cc} in place of Theorem \ref{th:cc}.
\end{proof}

\begin{remark}
When $k=1$, Theorem \ref{th:surface_extensions} is equivalent to the statement that the probability that $w_n$ is pseudo-Anosov goes to $1$ as $n \to \infty$. Versions of this result were proven by Rivin in \cite{rivin2008walks} and Maher in \cite{Maher}. Likewise, Theorem \ref{th:free_extensions} was also previously known in the special case where $k=1$, though by different methods. Indeed, Ilya Kapovich and Igor Rivin showed that for a random walk $(w_n)$ on $\Out(F)$ (with additional restrictions on the measure $\mu$), the probability that $w_n$ is atoroidal and fully irreducible goes to 1 as $n \to \infty$ \cite{rivin2010zariski}. (See Section \ref{sub:hyp_free} for definitions.)
\end{remark}

\subsection{Further applications}Thanks to the generality of Theorem \ref{th:random_qi}, which is made possible by the general framework of \cite{MaherTiozzo}, 
we can provide several other applications of interest. Following \cite{MaherTiozzo}, we say that $G$ is \emph{weakly hyperbolic} if $G$ admits a nonelementary action on a separable hyperbolic space $X$. 
We say that a \emph{random subgroup of $G$ has property $P$} if $$\mathbb{P}[\Gamma(n)\text{ has }P] \to 1 $$
as $n \to \infty$. 
For any nonelementary measure $\mu$, the proof of Theorem \ref{th:random_qi} additionally 
yields the following corollary.

\begin{corollary}\label{c:random_free}
A random subgroup of a weakly hyperbolic group is free and undistorted.
\end{corollary}

We note that when $G$ itself is a hyperbolic group, Gilman, Miasnikov, and Osin have shown that 
a random $k$-generated subgroup of $G$ is free and undistorted \cite{gilman2010exponentially}. 
The novelty of Corollary \ref{c:random_free} is that it applies to a much larger class of groups; e.g. 
mapping class groups, outer automorphism groups of free 
groups, right-angled Artin groups, acylindrically hyperbolic groups, and others.
Our model of random subgroups appears in several other places in the literature: 
starting from Guivarc'h's random walk proof of the Tits alternative \cite{guivarch1990prod}, it is
explicitly formulated by Rivin \cite{rivin2010zariski}, and Aoun \cite{aoun2011random}, who 
proves that a random subgroup of a non-virtually solvable linear group is free and undistorted.
Moreover, Myasnikov and Ushakov \cite{MU08crypto} prove that random subgroups of pure braid groups are free, 
and provide applications of such results to cryptography. \\

As in the situation of the actions $\Mod(S) \curvearrowright \C(S)$ and $\Out(F) \curvearrowright \I$ (discussed in Sections \ref{sub:hyp_surface} and \ref{sub:hyp_free}), one is often particularly interested in those elements that act as loxodromic isometries (i.e. have positive translation length). It is immediate from Theorem \ref{th:random_qi} that, under the hypotheses of the theorem, every element of a random subgroup of $G$ is loxodromic.

 In the special case of a right-angled Artin group $A(\Gamma)$, Kim and Koberda have introduced the extension graph $\Gamma^e$, a hyperbolic graph which admits a nonelementary action $A(\Gamma) \curvearrowright \Gamma^e$ \cite{KimKoberda2013}. The loxodromic isometries of $A(\Gamma)$ with respect to this action are precisely those $g\in A(\Gamma)$ with cyclic centralizers. In \cite{KMT1}, the authors study the properties of \emph{purely loxodromic} subgroups of $A(\Gamma)$, i.e. the subgroups for which all nontrivial elements are loxodromic, and show their resemblance to convex cocompact subgroups of mapping class groups. From Theorem \ref{th:random_qi}, we have the following:

\begin{corollary}\label{c:random_lox}
Let $\Gamma$ be a finite simplicial graph that does not decompose as a nontrivial join. Then a random subgroup of $A(\Gamma)$ is purely loxodromic.
\end{corollary}

\subsection*{Acknowledgments} We would like to thank Joseph Maher for several useful conversations. The first named author is partially supported by NSF DMS-1204592.

\section{Background}
\subsection{Random walks} \label{sub:random} 
Let $\mu$ be a probability measure on a countable group $G$. 
In order to define a random walk on $G$, let us consider  a probability space $(\Omega, \mathbb{P})$, 
and  for each $n \in \mathbb{N}$, let $g_n : \Omega \to G$ be a $G$-valued random variable
such that the $g_n$'s are independent and identically distributed with distribution $\mu$.
We define the $n^{th}$ step of the random walk to be the variable $w_n : \Omega \to G$ 
$$w_n := g_1 g_2 \dots g_n.$$
The sequence $(w_n)_{n \in \mathbb{N}}$ is called a \emph{sample path} of the random walk. We denote as $\mu_n$ the distribution of $w_n$, which equals the n-fold convolution of $\mu$ with itself. That is, 
$$ \mu_n(x)=\P[w_n = x] = \sum_{x=g_1\ldots g_n} \mu(g_1) \cdot \ldots \cdot \mu(g_n),$$
which describes the probability that the $n^{th}$ step of the random walk lands at $x \in G$.

Moreover, the \emph{reflected measure} $\check{\mu}$ is defined as $\check{\mu}(g) := \mu(g^{-1})$. We note that for the random walk $(w_n)$, the distribution of $w_n^{-1}$, the inverse of the $n^{th}$ step of the random walk, is given by $(\check\mu)_n = \check{(\mu_n)}$. 

Just as a random walk allows one to speak of a ``random'' element of $G$, we can use $k$ independent random walks to model a ($k$--generator) random subgroup of $G$. In details, fix a probability measure $\mu$ on $G$ and let $(w^1_n)_{n \in \N}, \ldots, (w^k_n)_{n \in \N}$ be $k$ independent random walks each of whose increments are distributed according to $\mu$. For each $n\in \N$, we can consider the subgroup generated by the $n^{th}$ steps of our sample paths,
$$\Gamma(n) = \langle w^1_n, \ldots, w^k_n \rangle \le G.$$ 
For example, if $G$ is finitely generated and $\mu$ is supported on a finite, symmetric generating set $S$ for $G$, then $\Gamma(n)$ is the subgroup generated by selecting $k$ (unreduced) words of length $n$ in the basis $S$ uniformly at random.

We remark that this model for a random subgroup of $G$ appears several places in the literature, 
see for example \cite{guivarch1990prod}, \cite{MU08crypto},  \cite{gilman2010exponentially}, \cite{rivin2010zariski}, and \cite{aoun2011random}.

Now suppose that $X$ is a separable hyperbolic space and that $G$ acts on $X$ by isometries. The action $G \curvearrowright X$ is \emph{nonelementary} if there are $g,h \in G$ which act loxodromically on $X$ and whose fixed point sets on the Gromov boundary of $X$ are disjoint. Recall that an isometry $g$ of $X$ is \emph{loxodromic}  if it has positive translation length on $X$, i.e. $\liminf_{n \to \infty} d(x_0,gx_0)/n > 0$ for some $x_0 \in X$. A probability measure $\mu$ on $G$ is said to be \emph{nonelementary} with respect to the action $G \curvearrowright X$ if the semigroup generated by the support of $\mu$ is a subgroup of $G$ whose action on $X$ is nonelementary. In this note, we are interested in the behavior of the image of a random walk $(w_n)$ under an orbit map $G \to X$. Hence, we fix once and for all a basepoint $x_0 \in X$ and consider the orbit map $G \to X$ given by $g \mapsto gx_0$. 

Central to Maher and Tiozzo's study of random walks on $G$ is the notion of shadows, which we now summarize. Given $x_0,x \in X$ and $R\ge0$, the \emph{shadow} $S_{x_0}(x,R) \subset X$ is by definition the set
$$S_{x_0}(x,R)  = \{y\in X: (x \cdot y)_{x_0} \ge d_X(x_0,x) -R\} .$$
When $R<0$, we declare that $S_{x_0}(x,R) = \emptyset$. The \emph{distance parameter} of the shadow $S_{x_0}(x,R)$ is the quantity $d(x_0,x) -R$, which is coarsely equal to the distance from $x_0$ to the shadow. It follows easily from hyperbolicity of $X$ that there is a constant $C$, depending only on the hyperbolicity constant of $X$, such that 
\begin{align}\label{ob:shadow}
(X \setminus S_{x_0}(x,R)) \subset S_x(x_0, d(x_0,x)-R+C).
\end{align} See \cite{MaherTiozzo} for details.

We will need two additional results from \cite{MaherTiozzo}. The first roughly states that the measure of a shadow decays to zero as the distance parameter goes to infinity. For the precise statement, set 
$$Sh(x_0,r) = \{S_{x_0}(gx_0,R) \ : \ g\in G \text{ and } d(x_0,gx_0) -R \ge r \}.$$
This is the set of shadows based at $x_0$ and centered at points in the orbit of $x_0$ with distance parameter at least $r$. The following is Corollary $5.3$ of \cite{MaherTiozzo}.

\begin{lemma}[Maher--Tiozzo] \label{lem:shadow_decay}
Let $G$ be a countable group which acts by isometries on a separable hyperbolic space $X$, and let $\mu$ be a nonelementary probability distribution on $G$. Then there is a function $f(r)$, with $f(r) \to 0$ as $r \to \infty$ such that for all $n$
\[
\sup_{S \in Sh(x_0,r)}
\left\{
\begin{array}{c} 
\mu_n(S)  \vspace{3pt}\\
\check{\mu}_n(S) 
\end{array} \right\} \le f(r).
\]
\end{lemma}

Finally, we will need the fact that a random walk on $G$ whose increments are distributed according to a nonelementary measure has positive drift in $X$. This is Theorem $1.2$ of \cite{MaherTiozzo}.
\begin{theorem}[Maher--Tiozzo] \label{thm:drift}
Let $G$ be a countable group which acts by isometries on a separable hyperbolic space $X$, and let $\mu$ be a nonelementary probability distribution on $G$. Fix $x_0\in X$. Then, there is a constant $L>0$ such that for almost every sample path
$$\liminf_{n \to \infty}\frac{d(x_0,w_nx_0)}{n} = L >0.$$
\end{theorem}
The constant $L >0$ in Theorem \ref{thm:drift} is called the \emph{drift} of the random walk $(w_n)$. 

\subsection{Hyperbolic extensions of surface groups}\label{sub:hyp_surface}
Here, we briefly recall some background on convex cocompact subgroups of mapping class groups. See \cite{FarbMosher, KentLein} for details.

Fix $S=S_g$, a closed, orientable surface of genus $g\ge2$. Associated to $S$ are its mapping class group $\Mod(S)$, its Teichm\"uller space $\T(S)$ (considered with the Teichm\"uller metric), and its curve graph $\C(S)$. We refer the reader to \cite{FM} for definitions and background on these standard objects in surface topology. We recall that there are natural actions $\Mod(S) \curvearrowright \T(S)$ and $\Mod(S) \curvearrowright \C(S)$; the former given by remarking and the latter given by the action of mapping classes on isotopy classes of simple closed curves (see remarks following Theorem \ref{th:cc} for details). While the action $\Mod(S) \curvearrowright \T(S)$ is properly discontinuous, $\T(S)$ is not negatively curved \cite{masur1975class, masur1994teichm}. On the other hand, $\C(S)$ is a locally infinite graph and the action $\Mod(S) \curvearrowright \C(S)$ has large vertex stabilizers, but $\C(S)$ is hyperbolic by the foundational work of Masur and Minsky \cite{MasurMinsky}. Much can be learned about the coarse geometry of $\Mod(S)$ by studying its action on both $\T(S)$ and $\C(S)$ in conjunction with the equivalent coarsely- Lipschitz map $\T(S) \to \C(S)$, which associates to each marked hyperbolic surface its collection of shortest curves.

In \cite{FarbMosher}, Farb and Mosher introduced convex cocompact subgroups of $\Mod(S)$ as those finitely generated subgroups $\Gamma \le \Mod(S)$ for which the orbit $\Gamma\cdot X \subset \T(S)$ of some $X \in \T(S)$ is quasiconvex with respect to the Teichm\"uller metric. Our interest in convex cocompact subgroups of $\Mod(S)$ comes from their connection to hyperbolicity of surface group extensions. This connection is summarized in the following theorem. For additional characterizations of convex cocompactness, see \cite{KentLein, DurhamTaylor}.

\begin{theorem}[Farb--Mosher, Kent--Leininger, Hamenst\"adt] \label{th:cc}
Let $\Gamma$ be a finitely generated subgroup of $\Mod(S)$. Then the following are equivalent
\begin{enumerate}
\item $\Gamma$ is convex cocompact,
\item the orbit map $\Gamma \to \C(S)$ is a quasi--isometric embedding, and
\item the extension $E_\Gamma$ is hyperbolic. 
\end{enumerate} 
\end{theorem}

In Theorem \ref{th:cc}, the implication $(3) \implies (1)$ is due to Farb--Mosher, as is the converse $(1) \implies (3)$ when $\Gamma$ is free \cite{FarbMosher}. The general case of $(1) \implies (3)$ is due to Hamenst\"adt \cite{H}. Finally, the equivalence $(1) \iff (2)$ is due to Kent--Leininger \cite{KentLein} and, independently, Hamenst\"adt \cite{H}. To show that random subgroups of $\Mod(S)$ induce hyperbolic extensions of $\pi_1(S)$ (Theorem \ref{th:surface_extensions}), we only need the implication $(2) \implies (3)$ in Theorem \ref{th:cc}. For this, it suffices to know a few details about the action $\Mod(S) \curvearrowright \C(S)$, which we summarize here.

The curve graph $\C(S)$ is the graph whose vertices are isotopy classes of essential simple closed curve and whose edges join vertices that have disjoint representatives on $S$. As stated above, Masur--Minsky showed that $\C(S)$ is hyperbolic and that the loxodromic elements of the action $\Mod(S) \curvearrowright \C(S)$, i.e. those elements with positive translation length, are precisely the pseudo-Anosov mapping classes \cite{MasurMinsky}. From this, it easily follows that the action $\Mod(S) \curvearrowright \C(S)$ is nonelementary. In fact, it is known that the action satisfies the much stronger property of being acylindrical \cite{Bo}, however, we will not need this fact here.

\subsection{Hyperbolic extensions of free groups}\label{sub:hyp_free}
Here, we recall some background on hyperbolic extensions of free groups. See \cite{DT1} for additional detail.

Fix $F =F_N$, the free group of rank $N\ge 3$. In \cite{DT1}, Dowdall and the first author study conditions on $\Gamma \le \Out(F)$ which imply that the extension $E_\Gamma$ is hyperbolic. In this note, we require a (possibly weaker) version of their main theorem. We begin by describing the particular $\Out(F)$ analog of the curve graph that we will require. This is a version of the intersection graph $\mathcal{I}$; an $\Out(F)$--graph introduced by Kapovich and Lustig in \cite{kapovich2007geometric}. 

First, let $\mathcal{I}'$ be the graph whose vertices are conjugacy classes of $F$ and two vertices are joined by an edge if there is a very small simplicial tree $F \curvearrowright T$ in which each conjugacy class fixes a point. (Recall that a simplicial tree is very small if edge stabilizers are maximal cyclic and tripod stabilizers are trivial.) Define $\mathcal{I}$ to be the connected component of $\mathcal{I}'$ that contains the primitive conjugacy classes, i.e. those conjugacy classes that belong to some basis for $F$. Note that we have a simplicial action $\Out(F) \curvearrowright \I$. The following theorem appears in \cite{DT1}, where it is attributed to Brian Mann and Patrick Reynolds:

\begin{theorem}[Mann--Reynolds \cite{MR2}]\label{th:M-R}
The graph $\mathcal{I}$ is hyperbolic and $f \in \Out(F)$ acts with positive translation length on $\mathcal{I}$ if and only if $f$ is atoroidal and fully irreducible. 
\end{theorem}

Although we will only require the statement of Theorem \ref{th:free_cc}, we recall that $f \in \Out(F)$ is fully irreducible if no positive power of $f$ fixes any conjugacy class of any free factor of $F$. Also, $f \in \Out(F)$ is atoroidal if no positive power fixes any conjugacy class of elements of $F$. If the extension $E_\Gamma$ of $\Gamma \le \Out(F)$ is hyperbolic then it is necessarily the case that each infinite order element of $\Gamma$ is atoroidal. The following appears as Theorem $9.2$ of \cite{DT1}.

\begin{theorem}[Dowdall--Taylor]\label{th:free_cc}
Let $\Gamma$ be a finitely generated subgroup of $\Out(F)$. Suppose that the orbit map $\Gamma \to \I$ is a quasi--isometric embedding for some $x_0 \in \I$. Then the corresponding extension $E_\Gamma$ is hyperbolic.
\end{theorem}

Just as in the situation of the mapping class group acting on the curve graph, it follows from Theorem \ref{th:M-R} that the action $\Out(F) \curvearrowright \I$ is nonelementary, i.e. there exists a pair loxodromic elements with no common fixed points on $\partial \I$. In fact, according to Mann--Reynolds \cite{MR2}, the action $\Out \curvearrowright \I$ satisfies the stronger property of being WPD.

\section{Proof of Theorem \ref{th:random_qi}}
We begin by providing conditions for when the orbit map from a $k$-generator group into a hyperbolic space is a quasi--isometric embedding. We require the following well-known lemma; see, for example, \cite{Gromov,GhysHarpe, gilman2010exponentially}. First, recall that for $x,y,z \in X$, the \emph{Gromov product}, denoted $(y \cdot z)_x$, is defined as 
$$(y \cdot z)_x = \frac{1}{2}(d(x,y)+d(x,z) - d(y,z)). $$
When $X$ is hyperbolic, the Gromov product $(y \cdot z)_z$ coarsely agrees with the distance from $x$ to any geodesic between $y$ and $x$. See \cite{BH} for details.

\begin{lemma}\label{lem:def_distance}
Let $X$ be a $\delta$-hyperbolic metric space with points $p_0, \ldots, p_n$ satisfying
\begin{align} \label{ineq:def_distance}
\min{\{d(p_{i-1},p_i) , d(p_i,p_{i+1})\}} \ge 2(p_{i-1}\cdot p_{i+1})_{p_i} + 18\delta +1.
\end{align}
Then, $d(p_0,p_n) \ge n$.
\end{lemma}

Lemma \ref{lem:def_distance} easily implies the following:
\begin{lemma}\label{lem:qi_conditions}
Let $X$ be a $\delta$-hyperbolic space and, for $1\le i\le k$, let $g_i \in \mathrm{Isom}(X)$ be isometries of $X$ such that for some $x_0\in X$ we have
\begin{align}\label{ineq:assumption}
d(x_0, g_ix_0) \ge 2(g_j^{\pm 1}x_0 \cdot g_l^{\pm 1}x_0)_{x_0} +18\delta +1
\end{align}
for all $1\le i,j,l \le k$ except when $j=l$ and the exponent on the $g_j$ and $g_l$ are the same.
Then the orbit map $ \langle g_1, \ldots, g_k \rangle \to X$ given by $g \mapsto g x_0$ is a quasi-isometric embedding.
\end{lemma}

\begin{proof}
Set $\Gamma = \langle g_1 \ldots, g_k \rangle$. As the orbit map $\Gamma \to X$ is always coarsely Lipschitz, it suffices to prove that for any $g \in \Gamma$,
$$|g|_\Gamma \le d(x_0,gx_0). $$
To see this, write $g= s_0 \ldots s_n$ as a reduced word where $s_i \in  \{g_0^{\pm 1}, \ldots, g_k^{\pm 1}\}$ and $n = |g|_\Gamma$. Letting $p_i = (s_0s_1\ldots s_i) x_0$, we note that by Lemma \ref{lem:def_distance} it suffices to show that Inequality \eqref{ineq:def_distance} holds for these points. Observe that since the action of $\Gamma$ is by isometries on $X$,
\begin{align*}
\min{\{d(p_{i-1},p_i) , d(p_i,p_{i+1})\}} &= \min\{d(x_0,s_ix_0), d(x_0,s_{i+1}x_0) \} \\
&\ge  2(s_i^{-1}x_0 \cdot s_{i+1} x_0)_{x_0} +18\delta +1\\
&= 2(p_{i-1} \cdot  p_{i+1})_{p_i} +18\delta +1
\end{align*}
where the first inequality holds by \eqref{ineq:assumption} and the fact that $s_i \neq s_{i+1}^{-1}$. This completes the proof.
\end{proof}

\begin{lemma}\label{lem:ind}
Let $G$ be a countable group with a nonelementary action by isometries on a hyperbolic space $X$. Let $\mu$ be a nonelementary probability measure on $G$ and fix $x_0 \in X$. Suppose that $(w_n)$ and $(u_n)$ be independent random walks on $G$ whose increments are distributed according to $\mu$. Then 
\begin{align}
\mathbb{P}[(w_n^{\pm1}x_0 \cdot u_n^{\pm1}x_0)_{x_0} \le l(n)] \to 1,  \label{ind}
\end{align}
as $n \to \infty$. Here, $l:\mathbb{N} \to \mathbb{N}$ is any function with $l(n)\to \infty$ as $n \to \infty$.
\end{lemma}

\begin{proof}
First note that since $(w_n)$ and $(u_n)$ are independent random walks with increments distributed according to $\mu$, both $w_n$ and $u_n$ have distribution $\mu_n$. Moreover, $w_n^{-1}$ and $u_n^{-1}$, the inverses of the $n^{th}$ steps, have distribution $\check{\mu}_n$ as noted at the beginning of Section \ref{sub:random}. Since the proofs of \eqref{ind} in each of the $4$ possible cases are identical, we show 
$$\mathbb{P}[(w_nx_0 \cdot u_n^{-1}x_0)_{x_0} \le l(n)] \to 1, $$
as $n \to \infty$.

By setting $R_n = d(x_0,u_nx_0) - l(n)$, we have
$$\mathbb{P}[(w_nx_0 \cdot u_n^{-1}x_0)_{x_0} \le l(n)] = 1 - \mathbb{P}[w_nx_0 \in S_{x_0}(u_n^{-1}x_0,R_n) ],$$
where the shadow $S_{x_0}(u_nx_0,R_n)$ has distance parameter $l(n)$. As $w_n$ and $u_n^{-1}$ are independent with distributions $\mu_n$ and $\check{\mu}_n$, respectively, we have that 
\begin{align*}
\mathbb{P}[w_nx_0 \in S_{x_0}(u_n^{-1}x_0,R_n) ] &= \sum_{g\in G}\mathbb{P}[w_nx_0 \in S_{x_0}(u_n^{-1}x_0,R_n) \ | \ u_n^{-1} =g ] \cdot \check{\mu}_n(g)  \\
&= \sum_{g\in G} \mu_n(S_{x_0}(gx_0,R_n)) \check{\mu}_n(g) \\
& \le f(l(n)),
\end{align*}
where the last inequality uses the decay of shadows (Lemma \ref{lem:shadow_decay}). Since $f(l(n)) \to 0$ as $n \to \infty$, the lemma follows.
\end{proof}

\begin{lemma}\label{lem:inverse}
Let $G$ be a countable group with a nonelementary action by isometries on a hyperbolic space $X$. Let $\mu$ be a nonelementary probability measure on $G$ and fix $x_0 \in X$. Suppose that $(w_n)$ is a random walk on $G$ whose increments are distributed according to $\mu$. Then 
$$\mathbb{P}[(w_nx_0 \cdot w_n^{-1}x_0)_{x_0} \le l(n)] \to 1, $$
as $n \to \infty$. Here, $l:\mathbb{N} \to \mathbb{N}$ is any function with $l(n)\to \infty$ as $n \to \infty$ and $\limsup \frac{l(n)}{n} < \frac{L}{2}$, 
where $L$ is the drift of $(w_n)$.
\end{lemma}

\begin{proof}
We follow the argument in \cite{MaherTiozzo}.
For each $n$, let $m := \lfloor \frac{n}{2} \rfloor$, and we can write  $w_n = w_m u_m$, with $u_m := w_m^{-1} w_n = g_{m+1}\dots g_n$.
Note that the random walks $w_m$ and $u_m$ are independent, and $u_m$ has distribution $\mu_{n-m}$. We first claim that 
$$\mathbb{P}[ (w_n x_0 \cdot w_m x_0)_{x_0} \ge l(n) ] \to 1 \qquad \textup{as }n \to \infty.$$
\begin{proof}[Proof of claim]
Indeed, 
\begin{align*}
\mathbb{P}[ (w_n x_0 \cdot w_m x_0)_{x_0} \ge l(n) ] & = \mathbb{P}[ (u_m x_0 \cdot x_0)_{w_m^{-1} x_0} \ge l(n) ] \\
& =  \mathbb{P}[u_m x_0 \in S_{w_m^{-1} x_0}(x_0,R) ] 
\end{align*}
where $R := d(x_0, w_m^{-1} x_0) - l(n)$. Using Observation \eqref{ob:shadow}, we note that
$$(X \setminus S_{w_m^{-1}x_0}(x_0, R))\subset S_{x_0}(w_m^{-1}x_0,R_n),$$
where $R_n = l(n) + C$, and $C$ depends only on the hyperbolicity constant of $X$. This implies that
\begin{align*}
\mathbb{P}[ (w_n x_0 \cdot w_m x_0)_{x_0} \ge l(n) ]  & \ge 1-  \mathbb{P}[u_m x_0 \in S_{ x_0}(w_m^{-1} x_0, R_n) ]. 
\end{align*}
Further, using independence of $w_m$ and $u_m$, we see
\begin{align}
\mathbb{P}[u_m x_0 \in S_{ x_0}(w_m^{-1} x_0, R_n) ]  & =\sum_{g\in G}\mathbb{P}[u_m x_0 \in S_{x_0}(w_m^{-1} x_0,R_n) \ | \  w_m =g ] \cdot \mu_m(g) \nonumber \\
&=  \sum_{g\in G}  \mu_{n-m}(S_{x_0}(g^{-1}x_0,R_n)) \mu_m(g).  \label{sum}
\end{align}
Let us now pick $\epsilon > 0$ such that $\limsup \frac{l(n) + \epsilon n}{n} < \frac{L}{2}$, where $L$ is the drift of $(w_n)$; then by considering in Equation \eqref{sum}
only the $g$ such that $d(x_0, g^{-1} x_0) \geq l(n) + \epsilon n$, and using the estimate for the distance parameter of $S_{x_0}(g^{-1}x_0,R_n)$
we get
\begin{align*}
\mathbb{P}[u_m x_0 \in S_{ x_0}(w_m^{-1} x_0, R_n) ] &\le f( \epsilon n - C ) +  \mathbb{P}[d(x_0, w_m x_0) \le l(n) + \epsilon n ]. 
\end{align*}
The first term tends to $0$ by the decay of shadows (Lemma \ref{lem:shadow_decay}) and the second because of linear progress (Theorem \ref{thm:drift}). This proves the claim.
\end{proof}

We return to the proof of Lemma \ref{lem:inverse}. 
As in the proof of the claim, replacing $w_n$ with $w_n^{-1}$ and $w_m$ with $w_n^{-1}w_m = u_m^{-1}$ it follows  that
$$\mathbb{P}[ (w_n^{-1} x_0 \cdot w_n^{-1}w_m x_0)_{x_0} \ge l(n) ] \to 1 \qquad \textup{as }n \to \infty.$$
Then, by Lemma \ref{lem:inverse} (using that $w_m$ and $u_m$ are independent)
$$\mathbb{P}[ (w_m x_0 \cdot w_n^{-1}w_m x_0)_{x_0} \le l(n)-3\delta ] \to 1 \qquad \textup{as }n \to \infty.$$
Finally, by $\delta$-hyperbolicity (Lemma \ref{lem:fellow_travel} below), 
$$\mathbb{P}[ (w_n x_0 \cdot w_n^{-1} x_0)_{x_0} \le l(n)+ 2\delta ] \to 1 \qquad \textup{as }n \to \infty.$$
This completes the proof.
\end{proof}

The following lemma was used in the proof of Lemma \ref{lem:inverse}. It appears as Lemma $5.9$ in \cite{MaherTiozzo}, but is proven here for convenience to the reader.
\begin{lemma}[Fellow traveling is contagious]\label{lem:fellow_travel}
Suppose that $X$ is a $\delta$--hyperbolic space with basepoint $x_0$ and suppose that $A\ge0$. If $a,b,c,d \in X$ are points of $X$ with $(a\cdot b)_{x_0} \ge A$, $(c\cdot d)_{x_0} \ge A$, and $(a\cdot c)_{x_0} \le A-3\delta$. Then $(b \cdot d)_{x_0} -2\delta \le (a \cdot c)_{x_0} \le (b \cdot d)_{x_0} +2\delta$.
\end{lemma}

\begin{proof}
By hyperbolicity, $(a\cdot c)_{x_0} \ge \min\{(a\cdot b)_{x_0},(b \cdot c)_{x_0} \} -\delta$. Since $(a \cdot c)_{x_0} \le (a \cdot b)_{x_0}- 3\delta$, it must be that $(a \cdot c)_{x_0} \ge (b \cdot c)_{x_0}  -\delta$. Exactly the same reasoning using the inequality $(a\cdot c)_{x_0} \ge \min\{(a\cdot d)_{x_0},(d \cdot c)_{x_0} \} -\delta$ gives that $(a\cdot c)_{x_0} \ge (a \cdot d)_{x_0} -\delta$. Hence, $(a \cdot d)_{x_0} \le (a\cdot c)_{x_0} +\delta \le (a \cdot b)_{x_0} -2\delta$.

Another application of hyperbolicity yields $(a \cdot d)_{x_0} \ge \min\{(a\cdot c)_{x_0},(c\cdot d)_{x_0}\} -\delta = (a\cdot c)_{x_0} -\delta$. Combining these facts we have
\begin{align*}
(b\cdot d)_{x_0} &\ge \min\{(a\cdot b)_{x_0}, (a\cdot d)_{x_0} \} -\delta \\
&=(a \cdot d)_{x_0} -\delta\\
&\ge (a \cdot c)_{x_0} - 2\delta.
\end{align*}

To prove the reverse inequality, first note that the inequality $(a \cdot d)_{x_0} \le (a \cdot b)_{x_0} -2\delta$ obtained above implies that 
$(a\cdot d)_{x_0} \ge \min\{(a\cdot b)_{x_0}, (b \cdot d)_{x_0} \} -\delta = (b\cdot d)_{x_0} -\delta$. Then
\begin{align*}
(a\cdot c)_{x_0} &\ge \min\{(a\cdot d)_{x_0}, (d\cdot c)_{x_0} \} -\delta\\
&= (a\cdot d)_{x_0} -\delta \\
&\ge (b\cdot d)_{x_0} -2\delta,
\end{align*}
as required.
\end{proof}

We can now complete the proof of Theorem \ref{th:random_qi}.
\begin{proof}[Proof of Theorem \ref{th:random_qi}]
Fix $x_0\in X$ and set $\Gamma(n) = \langle w_n^1, \ldots w_n^k \rangle$. By Lemma \ref{lem:qi_conditions},
the probability that $\Gamma(n) \to X$ is a quasi-isometric embedding is bounded below by the probability that  
\begin{equation} \label{to_one}
d(x_0, w_n^ix_0) \ge 2((w_n^j)^{\pm 1}x_0 \cdot (w_n^l)^{\pm 1}x_0)_{x_0} +18\delta +1
\end{equation}
for all choices of indices $1\le i,j,l \le k$, excluding the cases that produce terms involving the Gromov product of a point with itself.

It is easily verified that the probability of \eqref{to_one} goes to $1$ as $n \to \infty$. Indeed, by Theorem \ref{thm:drift},
$$\P[d(x_0, w_n^ix_0) \ge Ln] \to 1 $$
as $n\to \infty$, for each $1\le i \le k$, where $L$ is the drift of the random walks. Moreover, combining Lemma \ref{lem:ind} and Lemma \ref{lem:inverse}, we see that for $j \neq l$,
$$\P[((w_n^j)^{\pm 1}x_0 \cdot (w_n^l)^{\pm 1}x_0)_{x_0} \le l(n)] \to 1 $$
and for each $1 \le j \le k$,
$$\P[(w_n^jx_0 \cdot (w_n^j)^{-1}x_0)_{x_0} \le l(n)] \to 1, $$
for any function $l(n)$ with $l(n)\to \infty$ as $n \to \infty$ and $\limsup \frac{l(n)}{n} \le \frac{1}{2}L$. Additionally choosing $l(n)$ so that $2l(n) +18\delta +1 \le Ln$ completes the proof.
\end{proof}

\bibliography{random_extensions.bbl}

\providecommand{\bysame}{\leavevmode\hbox to3em{\hrulefill}\thinspace}
\providecommand{\MR}{\relax\ifhmode\unskip\space\fi MR }
\providecommand{\MRhref}[2]{%
  \href{http://www.ams.org/mathscinet-getitem?mr=#1}{#2}
}
\providecommand{\href}[2]{#2}
\begin{thebibliography}{GMO10}

\bibitem[Aou11]{aoun2011random}
Richard Aoun, \emph{Random subgroups of linear groups are free}, Duke
  Mathematical Journal \textbf{160} (2011), no.~1, 117--173.

\bibitem[BH09]{BH}
Martin~R. Bridson and Andre Haefliger, \emph{Metric spaces of non-positive
  curvature}, vol. 319, Springer, 2009.

\bibitem[Bir69]{birman1969mapping}
Joan~S. Birman, \emph{Mapping class groups and their relationship to braid
  groups}, Communications on Pure and Applied Mathematics \textbf{22} (1969),
  no.~2, 213--238.

\bibitem[Bow08]{Bo}
Brian~H. Bowditch, \emph{Tight geodesics in the curve complex}, Invent. Math.
  \textbf{171} (2008), no.~2, 281--300.

\bibitem[DKL12]{DKL1}
Spencer Dowdall, Richard~P. {Kent, IV}, and Christopher~J. Leininger,
  \emph{{P}seudo-{A}nosov subgroups of fibered 3-manifold groups}, arXiv
  preprint arXiv:1208.2495 (2012).

\bibitem[DT14a]{DT1}
Spencer Dowdall and Samuel~J. Taylor, \emph{Hyperbolic extensions of free
  groups}, arXiv preprint arXiv:1406.2567 (2014).

\bibitem[DT14b]{DurhamTaylor}
Matthew~Gentry Durham and Samuel~J. Taylor, \emph{Convex cocompactness and
  stability in mapping class groups}, preprint (2014), arxiv:1404.4803.

\bibitem[FM02]{FarbMosher}
Benson Farb and Lee Mosher, \emph{Convex cocompact subgroups of mapping class
  groups}, Geom. Topol. \textbf{6} (2002), 91--152 (electronic). \MR{1914566
  (2003i:20069)}

\bibitem[FM12]{FM}
Benson Farb and Dan Margalit, \emph{A primer on mapping class groups},
  Princeton Mathematical Series, vol.~49, Princeton University Press,
  Princeton, NJ, 2012.

\bibitem[GdlH90]{GhysHarpe}
Etienne Ghys and Pierre de~la Harpe, \emph{Sur les groupes hyperboliques
  d'apres mikhael gomov}, Birkhauser, 1990.

\bibitem[GMO10]{gilman2010exponentially}
Robert Gilman, Alexei Miasnikov, and Denis Osin, \emph{Exponentially generic
  subsets of groups}, Illinois Journal of Mathematics \textbf{54} (2010),
  no.~1, 371--388.

\bibitem[Gro87]{Gromov}
Mikhael Gromov, \emph{Hyperbolic groups}, Springer, 1987.

\bibitem[Gui90]{guivarch1990prod}
Yves Guivarc'h, \emph{Produits de matrices aleatoires et applications aux
  proprietes geometriques des sous-groupes du groupe lineaire}, Ergodic Theory
  Dynam. Sys. \textbf{10} (1990), no.~03, 483--512.

\bibitem[Ham05]{H}
Ursula Hamenst{\"a}dt, \emph{Word hyperbolic extensions of surface groups,
  arxiv:0807.4891v2}, 2005.

\bibitem[KK13]{KimKoberda2013}
Sang-hyun Kim and Thomas Koberda, \emph{Embedability between right-angled
  {A}rtin groups}, Geom. Topol. \textbf{17} (2013), no.~1, 493--530.

\bibitem[KL07]{KLSurvey}
Richard~P. {Kent, IV} and Christopher~J. Leininger, \emph{Subgroups of mapping
  class groups from the geometrical viewpoint}, In the tradition of
  {A}hlfors-{B}ers. {IV}, Contemp. Math., vol. 432, Amer. Math. Soc.,
  Providence, RI, 2007, pp.~119--141.

\bibitem[KL08]{KentLein}
Richard~P. Kent, IV and Christopher~J. Leininger, \emph{Shadows of mapping
  class groups: capturing convex cocompactness}, Geom. Funct. Anal. \textbf{18}
  (2008), no.~4, 1270--1325. \MR{2465691 (2009j:20056)}

\bibitem[KL09]{kapovich2007geometric}
Ilya Kapovich and Martin Lustig, \emph{Geometric intersection number and
  analogues of the curve complex for free groups}, Geom. Topol. \textbf{13}
  (2009), no.~3, 1805--1833.

\bibitem[KMT14]{KMT1}
Thomas Koberda, Johanna Mangahas, and Samuel~J Taylor, \emph{The geometry of
  purely loxodromic subgroups of right-angled {A}rtin groups}, arXiv preprint
  arXiv:1412.3663 (2014).

\bibitem[Mah11]{Maher}
Joseph Maher, \emph{Random walks on the mapping class group}, Duke Math. J.
  \textbf{156} (2011), no.~3, 429--468.

\bibitem[Mas75]{masur1975class}
Howard Masur, \emph{On a class of geodesics in {T}eichmuller space}, Annals of
  Mathematics (1975), 205--221.

\bibitem[MM99]{MasurMinsky}
Howard~A. Masur and Yair~N. Minsky, \emph{Geometry of the complex of curves.
  {I}. {H}yperbolicity}, Invent. Math. \textbf{138} (1999), no.~1, 103--149.

\bibitem[MR]{MR2}
Brian Mann and Patrick Reynolds, In preparation.

\bibitem[MT13]{MT1}
Johanna Mangahas and Samuel~J. Taylor, \emph{Convex cocompactness in mapping
  class groups via quasiconvexity in right-angled {A}rtin groups}, arXiv
  preprint arXiv:1306.5278 (2013).

\bibitem[MT14]{MaherTiozzo}
Joseph Maher and Giulio Tiozzo, \emph{Random walks on weakly hyperbolic
  groups}, arXiv preprint arXiv:1410.4173 (2014).

\bibitem[MU08]{MU08crypto}
Alexei Myasnikov and Alexander Ushakov, \emph{Random subgroups and analysis of
  the length-based and quotient attacks}, J. Math. Cryptol. \textbf{2} (2008),
  no.~1, 29--61.

\bibitem[MW94]{masur1994teichm}
Howard~A Masur and Michael Wolf, \emph{{T}eichmuller space is not {G}romov
  hyperbolic}, Math. Sci. Res. Inst. preprint (1994), no.~011-94.

\bibitem[Riv08]{rivin2008walks}
Igor Rivin, \emph{Walks on groups, counting reducible matrices, polynomials,
  and surface and free group automorphisms}, Duke Math. J. \textbf{142} (2008),
  no.~2, 353--379.

\bibitem[Riv10]{rivin2010zariski}
\bysame, \emph{Zariski density and genericity}, International Mathematics
  Research Notices (2010), rnq043.

\end{thebibliography}
\bibliographystyle{amsalpha}
\end{document}